\documentclass[12pt]{amsart}

\usepackage[all]{xy}
\usepackage{graphics}
\usepackage{color}
\usepackage[T1]{fontenc}
\usepackage{parskip}
\usepackage[colorlinks]{hyperref}
\usepackage{enumitem}
\usepackage{comment}

\usepackage{tikz}
\usepackage{tikz-cd}
\usetikzlibrary{decorations.markings}

\usepackage[colorinlistoftodos,prependcaption]{todonotes}

\newcommand{\B}[1]{{\mathbf #1}}
\newcommand{\C}[1]{{\mathcal #1}}

\newtheorem{theorem}{Theorem}[section]
\newtheorem{theorem*}{Theorem}

\newtheorem{proposition}[theorem]{Proposition}

\newtheorem*{question*}{Question}

\theoremstyle{definition}

\newtheorem{remark}[theorem]{Remark}
\newtheorem*{remark*}{Remark}
\newtheorem*{remarks*}{Remarks}
\newtheorem*{corollary*}{Corollary}

\numberwithin{figure}{section}
\numberwithin{table}{section}
\numberwithin{equation}{section}

\def\B{\mathbf}

\newcommand{\OP}{\operatorname}

\begin{document}

\title{On the entropy norm on $\OP{Ham}(S^2)$}
\author{Michael Brandenbursky}
\author{Egor Shelukhin}
\address{Ben Gurion University of the Negev, Israel}
\email{brandens@math.bgu.ac.il}
\address{University of Montreal, Canada}
\email{egorshel@gmail.com}
\subjclass[2000]{Primary 53; Secondary 57}

\begin{abstract}
In this note we prove that for each positive integer $m$ there exists a bi-Lipschitz embedding
$\B Z^m\to\OP{Ham}(S^2)$, where $\OP{Ham}(S^2)$ is equipped with the entropy metric. 
In particular, the same result holds when the entropy metric is substituted with the autonomous metric.
\end{abstract}

\maketitle


\section{Introduction}\label{S:intro} 
Let $S^2$ be the standard 2-sphere and $\OP{Ham}(S^2)$ the group of Hamiltonian diffeomorphisms of $S^2$.
There exist several unbounded bi-invariant metrics on $\OP{Ham}(S^2)$. The most notable are
the Hofer metric, the autonomous metric and the entropy metric, see e.g. \cite{MR3391653, BM-ent, MR99b:58037}. The following question
due to Kapovich and Polterovich is widely open and seems to be quite difficult:
"Is $\OP{Ham}(S^2)$ equipped with Hofer metric quasi-isometric to $\B R$?"

Let $\OP{Ent}(S^2)\subset \OP{Ham}(S^2)$ be the set of topological entropy-zero diffeomorphisms.
This set is conjugation invariant and it generates $\OP{Ham}(S^2)$, since this group is simple \cite{MR1445290}.  
In other words, each diffeomorphism in $\OP{Ham}(S^2)$ is a finite product of entropy-zero diffeomorphisms. 
One may ask for a minimal decomposition and this question
leads to the concept of the entropy norm which we define by
$$
\|f\|_{\OP{Ent}}:=\min\{k\in\B N\,|\,f=h_1\cdots h_k,\,h_i\in\OP{Ent}(S^2)\}.
$$
It is the word norm associated with the generating set $\OP{Ent}(S^2)$.
The entropy norm is conjugation-invariant since $\OP{Ent}(S^2)$ is. 
The associated bi-invariant metric is denoted by $\B d_{\OP{Ent}}$.

In this note we answer in negative the Kapovich-Polterovich question with respect to the entropy metric and hence
with respect to the autonomous metric. The main result of this paper is the following

\begin{theorem*}\label{T:Lip-emb}
For each $m\in\B N$ there exists a bi-Lipschitz embedding 
$$\B Z^m\hookrightarrow (\OP{Ham}(S^2), \B d_{\OP{Ent}}),$$
where $\B Z^m$ is endowed with the $l^1$-metric.
\end{theorem*}

In particular, the above result implies that $\OP{Ham}(S^2)$ equipped with either the entropy
or the autonomous metric is not a hyperbolic space.

\begin{remark}
The above theorem for $\OP{Ham}(D^2)$, where $D^2$ is a unit disc in the Euclidean plane, 
was recently proved in \cite{BM-ent}. Also, the above theorem for the Hofer metric and $\OP{Ham}(S_g)$,
where $S_g$ is a closed hyperbolic surface, was proved by Py in \cite{zbMATH05344052}.
For analogous results on other metrics on diffeomorphism groups of surfaces see e.g. \cite{2016-BS}.
\end{remark}

\subsection*{Acknowledgments.}
MB was partially supported by the Simons CRM Scholar-in-Residence Program.
MB wishes to express his gratitude to CRM for the support and excellent working conditions.

ES was supported by an NSERC Discovery Grant, 
by the Fonds de recherche du Quebec - Nature et technologies and Fondation Courtois.

\section{Preliminaries} \label{S:prelim}
In this section we recall the notion of a quasimorphism and describe
the Gambaudo-Ghys construction. Throughout the paper the area of $S^2$ is normalized to be one.

\subsection{Quasimorphisms} Recall that a function $\psi$ from a group $G$ to the reals
is called a quasimorphism if there exists a constant $D\geq 0$
such that $|\psi(g)-\psi(gh)+\psi(h)|<D$ for all $g,h\in G$. 
Minimal such $D$ is called the defect of $\psi$ and denoted by $D_\psi$.
A quasimorphism $\psi$ is homogeneous if $\psi(g^n)=n\psi(g)$ 
for all $n\in\B Z$ and $g\in G$. 
Quasimorphism $\psi$ can be homogenized by setting
$$\overline{\psi} (g) := \lim_{p\to\infty} \frac{\psi (g^p)}{p}.$$
The vector space of homogeneous quasimorphisms on $G$ is denoted by $Q(G)$. 
For more information about quasimorphisms and their 
connections to different branches of mathematics, see \cite{Calegari}.

\subsection{Gambaudo-Ghys construction} In their influential paper \cite{MR2104597} Gambaudo and Ghys 
constructed quasimorphisms on $\OP{Ham}(S)$, where $S$ is either a 2-disc or a 2-sphere, from
quasimorphisms on pure braid group $P_n$ or spherical pure braid group $P_n(S^2)$ respectively (see also \cite{BM-ent}).
The first named author generalized their construction to other surfaces \cite{MR3391653}.
Let us recall the construction.

Let $\{f_t\}$ be an isotopy in $\OP{Ham}(S^2)$ from the identity to $f\in\OP{Ham}(S^2)$
and let $z\in S^2$ be a basepoint. For $y\in S^2$ we define a loop
$\gamma_{y,z}\colon [0,1]\to S^2$ by
$$
\gamma_{y,z}(t):=
\begin{cases}
\alpha_{3t}(z) &\text{ for } t\in \left [0,\frac13\right ]\\
f_{3t-1}(y) &\text{ for } t\in \left [\frac13,\frac23\right ]\\
\beta_{3t-2}(f(y)) & \text{ for } t\in \left [\frac23,1\right ],
\end{cases}
$$
where $\{\alpha_t\}$ is the shortest path on $S^2$ from $z$ to $y$,
and $\{\beta_t\}$ is the shortest path on $S^2$ from $f(y)$ to $z$.

Let $\OP{X}_n(S^2)$ be the configuration space of all ordered $n$-tuples
of pairwise distinct points in $S^2$. Its fundamental group
$\pi_1(\OP{X}_n(S^2))$ is identified with the spherical pure braid group $P_n(S^2)$.
Let $z=(z_1,\ldots,z_n)$ in $\OP{X}_n(S^2)$ be a base point.
For almost every $x=(x_1,\ldots,x_n)\in \OP{X}_n(S^2)$ the
$n$-tuple of loops $(\gamma_{x_1,z_1},\ldots,\gamma_{x_n,z_n})$ is
a based loop in the configuration space $\OP{X}_n(S^2)$.
Let $\gamma(f,x)\in P_n(S^2)=\pi_1(\OP{X}_n(S^2),z)$
be the element represented by this loop, and let
$\varphi\colon P_n(S^2)\to \B R$ be a homogeneous quasimorphism.
Since $\pi_1(\OP{Ham}(S^2))$ is isomorphic to $\B Z_2$, the number
$\varphi(\gamma(f,x))$ does not depend on the choice of the isotopy $\{f_t\}$. 
Define the quasimorphism
$\Phi_n\colon \OP{Ham}(S^2)\to \B R$ and its homogenization
$\overline{\Phi}_n\colon\OP{Ham}(S^2)\to \B R$
by
\begin{equation}\label{eq:GG}
\Phi_n(f):=
\int\limits_{\OP{X}_n(S^2)}\varphi(\gamma(f;{x}))d{x}\qquad\qquad
\overline{\Phi}_n(f):=\lim_{p\to +\infty}\frac{\Phi_n(f^p)}{p}
\thinspace .
\end{equation}

\begin{remark}
The assertion that both the above functions are well defined quasimorphisms is
proved in \cite{MR2104597}.  Using the family of signature quasimorphisms
on the group $P_n(S^2)$ Gambaudo-Ghys showed that  
$\dim(Q(\OP{Ham}(S^2)))=\infty$. This fact was also proved in \cite{BM-ent}.
\end{remark}

\section{Proof of the main result} \label{S:main}

\begin{proposition}\label{P:disc}
Let $m,n\in\B N$ such that $n\geq 4$. Then there exist 
$f\in\OP{Ham}(S^2)$ supported in an embedded disc $D_m\subset S^2$ such
that $\OP{area}(D_m)<\frac{1}{m}$, and a quasimorphism $\overline{\Phi}_n$ as above 
such that $\overline{\Phi}_n(f)\neq 0$.
\end{proposition}

\begin{proof}
Let $m,n\in\B N$ such that $n\geq 4$, and let $\B X_n(S^2)$ 
be the configuration space of all unordered $n$-tuples of pairwise
distinct points in $S^2$. Recall that the Birman map:
$$
\OP{Push}\colon B_n(S^2)\to\OP{MCG}(S^2,n),
$$
where $B_n(S^2)=\pi_1(\B X_n(S^2),z)$ is the spherical braid group on $n$ strings
and $\OP{MCG}(S^2,n)$ is the mapping class group of the $n$-punctured sphere, 
is defined as follows: 
let $\alpha(t)$, $t\in[0,1]$, be a loop in $\B X_n(S^2)$ based at $z$ 
and $h_t\in\OP{Diff}(S^2)$ an isotopy such that $h_t(z)=\alpha(t)$. 
We define $\OP{Push}(\alpha):=[h_1]$ where $\alpha$ is the braid represented by the loop $\alpha(t)$.
The braid $\alpha$ is called \emph{reducible}
if $\OP{Push}(\alpha)$ is a reducible mapping class.

We denote by $Q_{\OP{BF}}(B_n(S^2))$ the space of homogeneous quasimorphisms
on $B_n(S^2)$ which vanish on reducible braids.
It follows from the celebrated paper by Bestvina and Fujiwara \cite{MR1914565} that
the space $Q_{\OP{BF}}(B_n(S^2))$ is infinite dimensional, see \cite[Section 4.A.]{BM-ent}.
Let $\varphi\in Q_{\OP{BF}}(B_n(S^2))$ and $g\in\OP{Ham}(S^2)$. 
Observe that if $x=(x_1,\ldots,x_n)\in \OP{X}_n(S^2)$ is such that there are $1\leq i<j\leq n$ so that
$x_i$ and $x_j$ lie outside of the support of $g$, then the braid $\gamma(g,x)$
is reducible and hence 

\begin{equation}\label{eq:reducible-braids}
\varphi(\gamma(g,x))=0.
\end{equation}

Let $\iota_n\colon P_n(S^2)\to B_n(S^2)$ be the standard inclusion. 
In \cite{MR3181631} Ishida proved that the composition map
$$
\C G_n\circ\iota_n^*\colon Q(B_n(S^2))\to Q(\OP{Ham}(S^2))
$$
is injective (see also \cite[Section 2.D.]{BM-ent}), 
where 
$$\C G_n\colon Q(P_n(S^2))\to Q(\OP{Ham}(S^2))$$ 
is the map defined by $\C G_n(\varphi)(f)=\overline{\Phi}_n(f)$, 
see equation \eqref{eq:GG}. In particular, the restriction
$$
\C G_n\circ\iota_n^*\colon Q_{\OP{BF}}(B_n(S^2))\to Q(\OP{Ham}(S^2))
$$
is injective. 

Let $\psi\in Q_{\OP{BF}}(B_n(S^2))$ a non-trivial quasimorphism. 
It follows from the paper of Ishida \cite{MR3181631} (see also \cite[Section 2.D.]{BM-ent})
that there exists an embedded disc $D_a\subset S^2$ of area $a$ (it could be very close to one)
and $f_a\in\OP{Ham}(S^2)$ such that the support of $f_a$ is contained in $D_a$
and $\overline{\Phi}_n(f_a)\neq 0$. It follows from equation \eqref{eq:reducible-braids} that
$$
\overline{\Phi}_n(f_a)=\lim_{p\to\infty}\left(\int\limits_{\OP{X}_n(D_a)}
\frac{\psi(\gamma(f^p;x))}{p}\thinspace dx+n(1-a)
\int\limits_{\OP{X}_{n-1}(D_a)}\frac{\psi(\gamma(f^p;x))}{p}\thinspace dx\right).
$$
Set $A:=\int\limits_{\OP{X}_n(D_a)}\displaystyle\lim_{p\to\infty}\frac{\psi(\gamma(f^p;x))}{p}\thinspace dx$
and $B:=\int\limits_{\OP{X}_{n-1}(D_a)}\displaystyle\lim_{p\to\infty}\frac{\psi(\gamma(f^p;x))}{p}\thinspace$.
Thus by the result of Ishida we have
\begin{equation}\label{eq:f_a}
\overline{\Phi}_n(f_a)=A+n(1-a)B\neq 0.
\end{equation}
Moreover, by the construction of Ishida, if we shrink the area of $D_a$ by $\varepsilon$ 
then we get a disc $D_{\varepsilon a}$
of area $\varepsilon a$ and $f_{\varepsilon a}\in\OP{Ham}(S^2)$ 
such that the support of $f_{\varepsilon a}$ is contained in $D_{\varepsilon a}$, and
\begin{equation}\label{eq:shrinking}
\overline{\Phi}_n(f_{\varepsilon a})=\varepsilon^n A+n\varepsilon^{n-1}(1-\varepsilon a)B=
\varepsilon^{n-1}(\varepsilon A+n(1-\varepsilon a)B).
\end{equation}
Note that if $B=0$, then by equation \eqref{eq:f_a} we get that $A\neq 0$,
and hence by equation \eqref{eq:shrinking} we get for each $\varepsilon$ 
that $\overline{\Phi}_n(f_{\varepsilon a})\neq 0$ and the proof follows.
If $B\neq 0$, than 
$$\lim_{\varepsilon\to 0}(\varepsilon A+n(1-\varepsilon a)B)=n B\neq 0.$$
It follows that there exists $\varepsilon$ such that $\varepsilon a<\frac{1}{m}$ and
$\overline{\Phi}_n(f_{\varepsilon a})\neq 0$. We set $D_m:=D_{\varepsilon a}$,
$f:=f_{\varepsilon a}$ and the proof follows.
\end{proof}

Let us continue the proof of the theorem.
It follows from \cite[Theorem 1]{BM-ent} that the subspace
$$\C G_n\circ\iota_n^*(Q_{\OP{BF}}(B_n(S^2)))\subset Q(\OP{Ham}(S^2))$$ 
is infinite dimensional for $n\geq 4$ and that every quasimorphism in this space 
vanishes on the set of entropy-zero diffeomorphisms. It follows from \cite[Lemma 3.10]{MR3044593}
and Proposition \ref{P:disc}
that there exists a family of quasimorphisms 
$\{\overline{\Phi}_{n,i}\}_{i=1}^m\in\C G_n\circ\iota_n^*(Q_{\OP{BF}}(B_n(S^2)))$
and a family of diffeomorphisms $\{f_{n,i}\}_{i=1}^m$ in $\OP{Ham}(S^2)$ such that the support of each $f_{n,i}$
is contained in a disc $D_m$ from Proposition \ref{P:disc}, and
$\overline{\Phi}_{n,i}(f_{n,j})=\delta_{ij}$, where $\delta_{ij}$ is the Kronecker delta.

Since $\OP{area}(D_m)<\frac{1}{m}$ there exists a family of diffeomorphisms 
$\{h_i\}_{i=1}^m$ in $\OP{Ham}(S^2)$
such that $h_i\circ f_{n,i}\circ h_i^{-1}$ and $h_j\circ f_{n,j}\circ h_j^{-1}$ 
have disjoint supports for $i\neq j$.
Denote by $\hat{f}_i:=h_i\circ f_{n,i}\circ h_i^{-1}$ and let
$$J\colon\B Z^m\to\OP{Ham}(S^2),$$
where $J(k_1,\ldots,k_m)=\hat{f}_1^{k_1}\ldots\hat{f}_m^{k_m}$.
It is clear that this map is a monomorphism. We prove that it is bi-Lipschitz.
Since all $\hat{f}_i$ commute with each other and $\overline{\Phi}_{n,i}(\hat{f}_j)=\delta_{ij}$,
we obtain
$$
\|\hat{f}_1^{k_1}\circ\ldots\circ\hat{f}_m^{k_m}\|_{\OP{Ent}}\geq\frac{|\overline{\Phi}_{n,i}(\hat{f}_1^{k_1}
\circ\ldots\circ\hat{f}_m^{k_m})|}{D_{\overline{\Phi}_{n,i}}}=\frac{|k_i|}{D_{\overline{\Phi}_{n,i}}}\thinspace,
$$
where $D_{\overline{\Phi}_{n,i}}$ is the defect of the quasimorphism
$\overline{\Phi}_{n,i}$. We denote by $\mathfrak{D}_m:=\max\limits_i D_{\overline{\Phi}_{n,i}}$
and obtain the following inequality
$$
\|\hat{f}_1^{k_1}\circ\ldots\circ\hat{f}_m^{k_m}\|_{\OP{Ent}}
\geq(m\cdot\mathfrak{D}_m)^{-1}\sum_{i=1}^m |k_i|\thinspace .
$$
Denote by $\mathfrak{M}_J:=\max\limits_i\|\hat{f}_i\|_{\OP{Ent}}$. 
Now we have the following inequality
$$
\|\hat{f}_1^{k_1}\circ\ldots\circ \hat{f}_m^{k_m}\|_{\OP{Ent}}
\leq\sum_{i=1}^m |k_i|\cdot\|\hat{f}_i\|_{\OP{Ent}}
\leq\mathfrak{M}_J\cdot\sum_{i=1}^m |k_i|\thinspace .
$$
Last two inequalities conclude the proof of the theorem. \qed

\bibliography{bibliography}
\bibliographystyle{plain}

\end{document}